\documentclass{amsart}
\pdfoutput=1
\usepackage{amscd,amssymb,amsopn,amsmath,amsthm,graphics,amsfonts,accents,enumerate,verbatim,calc}
\usepackage[dvips]{graphicx}
\usepackage{pgfplots}
\usepackage[all]{xy}
\usepackage{cite}
\usepackage{tikz}
\usepackage{tikz-cd}
\usepackage{mathrsfs}
\usepackage[pagebackref,colorlinks]{hyperref}
\hypersetup{
  pdftitle   = {The homotopy category of strongly flat modules},
  pdfauthor  = {Javad Asadollahi and Somayeh Sadeghi},
  pdfcreator = {},
  linkcolor  = red,
  citecolor  = blue
}

\calclayout

\newcommand{\rt}{\rightarrow}
\newcommand{\lrt}{\longrightarrow}

\newcommand{\st}{\stackrel}

\newcommand{\C}{\mathbb{C}}

\newcommand{\K}{\mathbb{K}}

\newcommand{\Z}{\mathbb{Z}}

\newcommand{\SA}{\mathscr{A}}

\newcommand{\SF}{\mathscr{F}}

\newcommand{\sS}{\mathscr{S}}
\newcommand{\ST}{\mathscr{T}}
\newcommand{\SU}{\mathscr{U}}

\newcommand{\SX}{\mathscr{X}}
\newcommand{\SY}{\mathscr{Y}}

 \newcommand{\SO}{\mathscr{O}}

\newcommand{\CF}{\mathcal{F} }

\newcommand{\CS}{\mathcal{S} }

\newcommand{\Mod}{{\rm{Mod\mbox{-}}}}

\newcommand{\Flat}{{\rm{Flat}}\mbox{-}}

\newcommand{\OC}{\mathbb{OC}}

\newcommand{\inc}{{\rm{inc}}}

\newcommand{\Add}{{\rm{Add}}}

\newcommand{\pd}{{\rm{pd}}}

\newcommand{\Ker}{{\rm{Ker}}}

\newcommand{\Prj}{{\rm{Prj}\mbox{-}}}

\newcommand{\SWC}{{S\rm{WC}}\mbox{-}R}
\newcommand{\SSF}{{S\rm{SF}}\mbox{-}R }
\newcommand{\SOS}{{\mathbb{O}S\rm{SF}}\mbox{-}R}

\newcommand{\Hom}{{\rm{Hom}}}
\newcommand{\homf}{{\mathcal{H}}om}
\newcommand{\Ext}{{\rm{Ext}}}

\theoremstyle{plain}
\newtheorem{theorem}{Theorem}[section]
\newtheorem{corollary}[theorem]{Corollary}
\newtheorem{lemma}[theorem]{Lemma}
\newtheorem{facts}[theorem]{Facts}
\newtheorem{proposition}[theorem]{Proposition}

\theoremstyle{definition}
\newtheorem{definition}[theorem]{Definition}
\newtheorem{example}[theorem]{Example}

\newtheorem{remark}[theorem]{Remark}

\theoremstyle{plain}

\theoremstyle{definition}

\numberwithin{equation}{section}

\begin{document}

\title[The Homotopy Category of Strongly flat modules]{The homotopy category of strongly flat modules}

\author[J. Asadollahi and S. Sadeghi ]{ Javad Asadollahi and Somayeh Sadeghi}

\address{Department of Pure Mathematics, Faculty of Mathematics and Statistics, University of Isfahan, P.O.Box: 81746-73441, Isfahan, Iran}
\email{asadollahi@sci.ui.ac.ir, asadollahi@ipm.ir }

\address{School of Mathematics, Institute for Research in Fundamental Sciences (IPM), P.O.Box: 19395-5746, Tehran, Iran}
\email{somayeh.sadeghi@ipm.ir }

\makeatletter \@namedef{subjclassname@2020}{\textup{2020} Mathematics Subject Classification} \makeatother

\subjclass[2020]{13B30, 13D05, 18N55, 18A40, 18G80, 16G20}

\keywords{Strongly flat modules, homotopy categories, well generated triangulated categories}

\begin{abstract}
In this paper, we plan to build upon significant results by Amnon Neeman regarding the homotopy category of flat modules to study ${\mathbb{K}}({S\rm{SF}}\mbox{-}R)$, the homotopy category of $S$-strongly flat modules, where $S$ is a multiplicatively closed subset of a commutative ring $R$. The category ${\mathbb{K}}({S\rm{SF}}\mbox{-}R)$ is an intermediate triangulated category that includes ${\mathbb{K}}({\rm{Prj}\mbox{-}} R)$, the homotopy category of projective $R$-modules, which is always well generated by a result of Neeman, and is included in ${\mathbb{K}}({\rm{Flat}}\mbox{-} R)$, the homotopy category of flat $R$-modules, which is well generated if and only if $R$ is perfect, by a result of \v{S}\'{t}ov\'{i}\v{c}ek. We analyze corresponding inclusion functors and the existence of their adjoints. In this way, we provide a new, fully faithful embedding of the homotopy category of projectives to the homotopy category of $S$-strongly flat modules. We introduce the notion of $S$-almost well generated triangulated categories. If $R$ is an $S$-almost perfect ring, ${\mathbb{K}}({\rm{Flat}}\mbox{-} R)$ is $S$-almost well generated. We show that the converse is true under certain conditions on the ring $R$. We hope that this approach provides insights into the largely mysterious class of $S$-strongly flat modules.
\end{abstract}

\maketitle


\section{Introduction}
Throughout, $R$ is a commutative ring with identity, $S \subset R$  is a multiplicative subset which may contain some zero-divisors and $R_S$ denotes the localization of $R$ with respect to $S$. An $R$-module $C$ is called $S$-weakly cotorsion if $\Ext^1_R(R_S, C)=0$. An $R$-module $F$  is $S$-strongly flat if $\Ext^1_R(F, C)=0$, for every $S$-weakly cotorsion module $C$. We let $\SWC$, resp. $\SSF$, denote the class of $S$-weakly cotorsion, resp. $S$-strongly flat, $R$-modules. These classes of modules are introduced in \cite[Section 2]{Tr1}, when $R$ is a commutative integral domain and $R_S$ is the quotient field of $R$.

It is known  that the pair $(\SSF, \SWC)$  is a complete cotorsion pair. In particular,  $\SSF$ is a precovering class and $\SWC$ is a preenveloping class.  This leads to the natural question of  when $\SSF$ is a covering class, a question raised in \cite{Tr2}. In \cite{BS} it is demonstrated that if $R$ is an integral domain and $R_S$ is the quotient field of $R$, then $\SSF$ is a covering class if and only if  $R$ is almost perfect, meaning all proper quotients of $R$ are perfect.   This result is further generalized in \cite{FS} to an arbitrary commutative ring $R$ with zero-divisors where $R_S$ is the total ring of quotients. A more extensive generalization for the case when $S$ may contain zero-divisors is discussed in \cite{BP}. It is shown that all $R$-modules possess $S$-strongly flat covers if and only if $R$ is an $S$-almost perfect ring, i.e.  the localization $R_S$ is a perfect ring and, for every $s \in S$, the quotient ring $R/sR$ is also perfect. This implies that all flat $R$-modules are $S$-strongly flat.

The structure of strongly flat modules has been studied for over two decades, but it remains somewhat mysterious \cite{FS}, \cite{P}, \cite{PS1}, \cite{BP} and \cite{FN}. In \cite{PS1}, the authors proposed a conjecture that they called the Optimistic Conjecture, denoted as $(\OC)$. This conjecture states that if the projective dimension of $R_S$ as an $R$-module doesn't exceed 1, then a flat $R$-module $F$ is $S$-strongly flat if the localized module $S^{-1}F = F_S$ is a projective $R_S$-module, and for every $s \in S$, the quotient module $F/sF$ is a projective $R/sR$-module. For the situations in which the conjecture $(\OC)$ is confirmed to be true, please refer to Remark \ref{Remark 2.2}$(2)$.

On the other hand, the theory of compactly generated triangulated categories implicitly known in algebraic topology, introduced by Neeman and used to simplify proofs of some classical results, \cite{Ne1}, \cite{ Ne2}. Then it has found many applications in algebra. But in both algebra and topology, there are many triangulated categories that are not compactly generated. However, these categories still share many important properties with compactly generated ones. To better understand this, Neeman defined well generated categories. A triangulated category $\ST$, with coproducts, is well generated if it has an $\alpha$-perfect set of $\alpha$-small generators, for some regular cardinal $\alpha$. Such a set is called an $\alpha$-compact generating set of $\ST$ \cite{N2}. In \cite{N1}, Neeman proved that for every ring $R$, the homotopy category of complexes of projective $R$-modules $\K(\Prj R)$  is always  well generated and hence satisfies Brown representability. One specific consequence of Brown's representability is the existence of certain adjoint functors. So the adjoint pairs of functors
 \[ \xymatrix{\K(\Prj R)\ar@<0.5ex>[rr]^{j_!}& &\K(\Flat R)\ar@<0.5ex>[ll]^{j^*} } \]
always exists, where $\K(\Flat R)$ is the homotopy category of complexes of flat $R$-modules.  Therefore,  there is a new description for the category $\K(\Prj R)$. More precisely, the category $\K(\Prj R)$ is equivalent to the Verdier  quotient  $\K(\Flat R)/\CS$, for some suitable subcategory $\CS$ of $\K(\Flat R)$, \cite[Remark 2.12]{N1}.
Several characterizations of the objects in the subcategory $\CS\subset\K(\Flat R)$ are presented in \cite[Facts 2.14]{N1}. Moreover, in \cite{N}, it is shown that the quotient map from $\K(\Flat R)$ to $\K(\Flat R)/\CS$ always has a right adjoint. So this gives a new fully faithful embedding of $\K(\Prj R)$ into $\K(\Flat R)$.

Motivated by the results of Neeman, that briefly recalled above, we study the homotopy category of $S$-strongly flat $R$-modules, where $R$ is a commutative ring and $S$ is a multiplicative subset of $R$. To this end, we introduce and study $\SOS$, the class of optimistically $S$-strongly flat $R$-modules. A flat $R$-module $F$ is  optimistically $S$-strongly flat if $F_S$ is a projective $R_S$-module and, for every $s\in S$, $F/sF$ is a projective $R/sR$-module. It is clear that every $S$-strongly flat $R$-module is optimistically $S$-strongly flat. Further,  if optimistic conjecture holds, the class of optimistically $S$-strongly flat modules is precisely the class of $S$-strongly flat modules. In case the multiplicative subset $S$ consists of powers of a single element $r \in R$, the class of optimistically $S$-strongly flat modules is equal to the class of $r$-very flat modules, that introduced and studied in deep in \cite{PS2}.

We use this class to show that the quotient map $\K(\SSF)\lrt \K(\SSF)/\sS$ has a right adjoint, provided optimistic conjecture holds, where $\sS$ is the full subcategory of $\K(\SSF)$ consisting of all objects annihilated by the functor $j^*$. This gives a new fully faithful embedding of $\K(\Prj R)$ into $\K(\SSF)$. Finally, we introduce the notion of $S$-almost well generated triangulated categories and demonstrate when the category $\K(\Flat R)$ is $S$-almost well generated.
Let us mention that in our proofs we have used the important fact proved by Neeman \cite[Remark 2.15]{N1} that every acyclic complex of projective modules with flat kernels is contractible.

The paper is structured as follows. In Section \ref{Sec: Preliminaries}  we recall and collect some known facts and results  specially about $S$-strongly flat modules that we need through the paper.

 In Section \ref{Sec: Strongly Flat Precover}, motivated by the work of Enochs and Garc\'{i}a Rozas \cite{EG},  we introduce and study the notions of $S$-strongly flat complexes. An acyclic complex $(F, \delta)$ is called $S$-strongly flat if, for every $i$, all kernels of the maps $\delta^i: F^i\lrt F^{i+1}$ are $S$-strongly flat. By definition it is clear that every $S$-strongly flat complex is a flat complex, that is, an acyclic complex with all kernels flat $R$-modules. In \cite{AEGO}, it is proved that every complex over a ring has a flat cover.  We show that if  projective dimension of $R_S$ as an $R$-module doesn't exceed $1$, then every complex over a ring $R$ has an $S$-strongly flat precover. We also show that the class of $S$-strongly flat complexes is covering if and only if $R$ is $S$-almost perfect (Theorem \ref{Covering}).

 Section \ref{Sec: Homotopy Category of Optimistically Strongly Flat modules} is the main section of this paper. We introduce the class of optimistically $S$-strongly flat $R$-modules $\SOS$ and study its homotopy category. We show that the natural inclusion  $e: \K(\Prj R)\lrt \K(\SOS)$  has a right adjoint $e^*: \K(\SOS)\lrt \K(\Prj R)$ (Proposition \ref{Right Adjoint}).
This, in turn, provides a description for $\K(\Prj R)$ as a Verdier quotient of $\K(\SOS)$. We study the kernel of $e^*$ and show that it consist of all acyclic complexes of modules in $\SOS$ whose all kernels are in $\SOS$ (Theorem \ref{Theorem 8.6}). Such complexes will be called, optimistically $S$-strongly flat complexes.
If we assume that optimistic conjecture holds for the pair $(R, S)$, these are $S$-strongly flat complexes.

Toward the end of this section, it is shown that the category $\K(\SOS)$ is a tensor triangulated category (Corollary \ref{Closed Under Tensor}).

In Section \ref{Sec:  Some Adjoints in Homotopy Category} we show that if optimistic conjecture holds for the pair $(R, S)$, then the functor $e^*:\K(\SSF)\lrt \K(\Prj R)$ has a right adjoint $e_*:\K(\Prj R)\lrt \K(\SSF)$ (Theorem \ref{Existence of adjoint of quotient}). Hence $e_*$ gives a non-obvious embedding of $\K(\Prj R)$ into $\K(\SSF)$.  Moreover, we observe that the inclusion functor $\K(\SSF)\lrt \K(\Mod R)$ has a right adjoint (Theorem \ref{Existence of adjoint of inclusion}).

Neeman in \cite{N}, in order to prove that every object in $\K(\Flat R)$ admits a flat cover,  defined an auxiliary ring $T=T(R)$ and an adjoint pair of functors $(\inc, C)$, where $\inc$ is a fully faithful functor from the category of complexes of $R$-module $\C(\Mod R)$ to the category of $T$-modules. In the rest of  Section \ref{Sec: Some  Adjoints in Homotopy Category}, we briefly recall the ring $T$ and functors $\inc$ and $C$ and their properties as well.  Then we study the behaviour of the class of $S$-strongly flat complexes and the class of optimistically $S$-strongly flat complexes under the functors $\inc$ and $C$.

 In the last section of the paper, we introduce the notion of  $S$-almost well generated triangulated categories and provide some examples (Example \ref{Examples}). For instance, $\K(\SSF)$ is always $S$-almost well generated. We show that if $R$ is an $S$-almost perfect ring then $\K(\Flat R)$ is $S$-almost well generated (Theorem \ref{S-Almost WG}). Further, we show the converse in a special case. For instance, let $R$ be a commutative Noetherian ring of Krull dimension $1$ and $S\subset R$ be a multiplicative subset consisting of some nonzero-divisors in $R$. Then $R$ is $S$-almost perfect, provided $\K(\Flat R)$ is $S$-almost well generated (Theorem \ref{Krull Dimension}).

\section{Preliminaries}\label{Sec: Preliminaries}
 Let $\SA$ be an abelian category. Let $\SX$ be the class of objects in $\SA$. An $\SX$-precover of an object $M\in \SA$, is a morphism $f:X\lrt M$ with $X\in\SX$, such that, for every $X'\in\SX$, the induced morphism $\Hom(X', X)\lrt\Hom(X', M)\lrt 0$ is exact. If, further, for every $g: X\lrt X$, $f\circ g=f$ implies that $g$ is an isomorphism, then $f:X\lrt M$ is an $\SX$-cover. A class $\SX$ is precovering, resp. covering, if every object in $\SA$ has an $\SX$-precover, resp. $\SX$-cover. Dually, we can define the notions of $\SX$-preenvelope, $\SX$-envelope, preenveloping and enveloping classes.

Recall that a pair $(\SX, \SY)$ of classes of objects of $\SA$ is a cotorsion pair if $\SX^{\perp_1}=\SY$ and ${}^{\perp_1}\SY=\SX$, where orthogonal is taken with respect to $\Ext^1$.

A cotorsion pair $(\SX, \SY)$ is complete if for every object $M$ in $\SA$, there are short exact sequences
\[0 \rt Y \rt X \rt M \rt 0  \ \ {\rm{and}} \ \ 0 \rt M \rt Y' \rt X' \rt 0,\]
such that $X, X' \in \SX$ and $Y, Y' \in \SY.$ This, in particular, implies that $\SX$ is a precovering class and $\SY$ is a preenveloping class.

A cotorsion pair $(\SX, \SY)$ is hereditary if $\SX$ is closed under taking kernels of epimorphisms between objects of $\SX$ and $\SY$ is closed under taking cokernels of monomorphisms between objects of $\SY$.

\subsection{Strongly flat modules}\label{Subsec: StronglyFlat}
As it is mentioned in the introduction, $\SWC$ is the right orthogonal of $R_S$  and $\SSF$ is the left orthogonal of $\SWC$, both orthogonals are with respect to $\Ext^1_R$. It is clear that every cotorsion $R$-module is $S$-weakly cotorsion and every $S$-strongly flat $R$-module is flat.

In the following, we collect some facts about these two classes of $R$-modules that we need through the paper. Recall that a ring $R$ is perfect if every flat $R$-module is projective. $R$ is called $S$-almost perfect if $R_S$ is a perfect ring and for every $s\in S$, $R/sR$ is a perfect ring.

\begin{facts}
The following hold true.
\begin{itemize}
\item[$(i)$] The class $\SSF$ is closed under extensions, coproducts and direct summands.

\item [$(ii)$] If $F$ is an $S$-strongly flat $R$-module, then the localization $F_S$ of $F$ at $S$ is projective as $R_S$-module and, for every $s \in S$, the quotient module $F/sF$ is projective $R/sR$-module \cite[Lemma 3.1]{BP}.

\item[$(iii)$] If projective dimension of $R_S$ as an $R$-module doesn't exceed $1$,  then the projective dimension of every $S$-strongly flat $R$-module also doesn't exceed $1$.

\item [$(iv)$] The class $\SWC$ is an enveloping class. The class $\SSF$ is covering if and only if every flat $R$-module is $S$-strongly flat if and only if $R$ is an $S$-almost perfect ring \cite[Theorem 7.9]{BP}.
\end{itemize}
\end{facts}

\begin{remark}\label{Remark 2.2}
$(1)$. A special $\SSF$-precover of a module $M$ is an $\SSF$-precover whose kernel is $S$-weakly cotorsion. Since $(\SSF, \SWC)$ is a complete cotorsion pair, every module admits a special $\SSF$-precover. Additionally, by the same reason, we infer that every module admits an special $\SWC$-preenvelope.

$(2)$. The validity of the converse of the Statement $(ii)$ of the above facts, is known as the Optimistic Conjecture $(\OC)$ \cite[1.1]{PS1}. More precisely, $(\OC)$, reads as follows:

\textit{Optimistic Conjecture}. Let $R$ be a commutative ring and $S$ be a multiplicative subset of $R$ such that the projective dimension of $R_S$ as an $R$-module doesn't exceed $1$. Then a flat $R$-module $F$ is $S$-strongly flat if and only if the $R_S$-module $F_S$ is projective and the $R/sR$-module $F/sF$ is projective for every $s\in S$.

Under the following conditions on the pair $(R, S)$, the $(\OC)$ holds.
\begin{itemize}
\item[$(i)$] $S$ is a countable multiplicative subset, \cite[Theorem 1.3]{PS1}.

\item[$(ii)$] $S$ is a multiplicative subset consisting of (some) nonzero-divisors in $R$ and projective dimension of $R_S$ as an $R$-module doesn't exceed $1$, \cite[Theorem 1.4]{PS1}.
For instance, if $R$ is a Matlis domain, that is a commutative domain such that projective dimension of $R_S$ as an $R$-module is at most $1$, where $S=R\smallsetminus \lbrace 0\rbrace$, then $(\OC)$ is valid.

\item[$(iii)$] $S$ is a multiplicative subset such that the $S$-torsion in $R$ is bounded and projective dimension of $R_S$ as an $R$-module doesn't exceed $1$, \cite[Theorem 1.5]{PS1}.

\item[$(iv)$] $R$ is an $S$-$h$-nil ring. That  is,  for every $s\in S$, the ring $R/sR$ is semilocal of  Krull dimension 0, \cite[Proposition 7.13]{BP}. For instance, if for every $s\in S$, $R/sR$ is a perfect ring, then $R$ is $S$-$h$-nil, \cite[Lemma 7.1]{BP}.
\end{itemize}
 \end{remark}

\subsection{Complexes of modules}\label{Subsec: Complex}
Let $R$ be a ring. The category of complexes in $\Mod R$ is denoted by $\C(\Mod R)$ and the homotopy category of complexes of $R$-modules is denoted by $\K(\Mod R)$.

A complex $(X, \delta)$ in $\C(\Mod R)$ will be denoted cohomologically, i.e.
\[\begin{tikzcd}
 (X, \delta): &\cdots\rar& X^{i-1}\rar{\delta^{i-1}}&X^{i}\rar{\delta^i}& X^{i+1} \rar& \cdots.
	\end{tikzcd}\]

It is known that the category $\C(\Mod R)$ is an abelian category with enough projective and enough injective objects.
Let $(X, \delta_X)$ and $(Y, \delta_Y)$ be two complexes of $R$-modules. We let $\homf(X, Y)$ to be the complex of abelian groups with
\[\homf(X, Y)^n=\prod_{i\in \mathbb{Z}} \Hom(X^i, Y^{n+i})\]
and if $f\in \homf(X, Y)^n$ then
\[(\delta^n f)^{i}=\delta_Y^{n+i}\circ f^i-(-1)^n f^{i+1}\circ \delta^i_X.\]
Moreover, we let $X\otimes Y$ be the complex of abelian groups with
\[(X\otimes Y)^n= \bigoplus _{i+j=n} X^i\otimes_R Y^j\]
and whose differential is given on elements $(x,y)$ of homogeneous degree by the following formula
\[\delta_{X\otimes Y}(x, y)=(\delta_Xx, y)+(-1)^{|x|}(x, \delta_Y y).\]

\section{Strongly flat precover of complexes}\label{Sec: Strongly Flat Precover}
  In  \cite{EG} the authors defined and studied flat complexes over any ring. By Definition 2.5 of \cite{EG}, a flat complex is an acyclic complex of flat $R$-modules such that all of its kernels are flat. Motivated by this definition, we introduce and study $S$-strongly flat complexes and prove that, if the projective dimension of $R_S$ as an  $R$-module doesn't exceed $1$, then any complex over a ring $R$ has an $S$-strongly flat precover. Further, we show that every complex has an $S$-strongly flat cover if and only if $R$ is an $S$-almost perfect ring.

We start with the following definitions.

\begin{definition}\label{DefinitionComplexes}
A complex $(F, \delta)$  of $R$-modules is called $S$-strongly flat  if it is acyclic and for every $i \in \Z$, the kernel $K^i$ of the map $\delta^i: F^i\lrt F^{i+1}$ is  $S$-strongly flat. Similarly, an acyclic complex $(C, \delta)$ is called  $S$-weakly cotorsion if all the kernels are $S$-weakly cotorsion $R$-modules.
\end{definition}

\begin{remark}\label{Properties of StronglyFlat Complexes}
 $(i)$ Since the class of $S$-strongly flat $R$-modules is closed under extensions, then in an $S$-strongly flat complex $F$ each module $F^i$ is $S$-strongly flat. Moreover, since $\SSF \subset \Flat R$, then every $S$-strongly flat complex is a flat complex.

$(ii)$ Since the class of $S$-weakly cotorsion $R$-modules is closed under extensions, then in an $S$-weakly cotorsion complex $C$ each module $C^i$ is $S$-weakly cotorsion. Also, since every (Enochs') cotorsion $R$-module is $S$-weakly cotorsion, every cotorsion complex, i.e. every acyclic complex with cotorsion kernels \cite[Definition 3.3 (ii)]{EG}, is $S$-weakly cotorsion.

\end{remark}

Recall that a complex $(X, \delta)$ of $R$-modules has projective dimension at most $n$ in $\C(\Mod R)$ if and only if $X$ is acyclic and projective dimension of $\Ker \delta^i$, for every $i$, is at most $n$, see \cite[Theorem]{EG1}.

\begin{lemma}\label{ProjDim}
Assume that the projective dimension of $R_S$ as an $R$-module doesn't exceed $1$. Then for every $S$-strongly flat complex $(F, \delta)$, the projective dimension of complex $F$ is at most $1$.
\end{lemma}

\begin{proof}
Let $F$ be an $S$-strongly flat complex, so it is acyclic and for every $i$, $\Ker\delta^i$ is $S$-strongly flat $R$-module.  Since $\pd_R R_S\leq 1$, projective dimension of any $S$-strongly flat $R$-module is at most $1$. Therefore $\pd_R \Ker\delta^i\leq 1$. Hence we get the result.
\end{proof}

\begin{lemma}\label{lem. 4.2}
Let  $R$ be a ring such that any acyclic complex has an $S$-strongly flat precover. Then any complex has an $S$-strongly flat precover.
\end{lemma}

\begin{proof}
By \cite[Theorem 3.18]{EJX} we know that for any complex $C$ there is an acyclic cover $E\lrt C$. By assumption, acyclic complex $E$ has an $S$-strongly flat precover $F\lrt E$. Now it is easy to see that the composition $F\lrt E\lrt C$ is an $S$-strongly flat precover of $C$.
\end{proof}

\begin{proposition}\label{Prop. 4.3}
 Let $(C, \delta)$ be an $S$-weakly cotorsion complex. Then, $C$ has an $S$-strongly flat precover, and its kernel is $S$-weakly cotorsion.
\end{proposition}

\begin{proof}
Since $C$ is $S$-weakly cotorsion, for every $i$, there exists a short exact sequence
\[\eta: \ 0\lrt \Ker\delta^i\lrt C^i\lrt \Ker \delta^{i+1}\lrt 0,\]
where $\Ker\delta^i$ and $\Ker\delta^{i+1}$ are $S$-weakly cotorsion. Consider special $S$-strongly flat precovers $F^i\lrt \Ker\delta^i$ and $F^{i+1}\lrt \Ker \delta^{i+1}$. Since $\Ext_R^1(F^{i+1}, \Ker\delta^{i})=0$, we can construct the commutative diagram
\[\begin{tikzcd}
 &	0 \rar  &F^i\rar\dar&F^i\oplus F^{i+1}\rar\dar&F^{i+1}\dar\rar& 0\\
 &	0\rar&\Ker\delta^i\rar &C^i\rar& \Ker\delta^{i+1}\rar & 0
\end{tikzcd}\]
It is easy to see that $0\lrt F^i\lrt F^i\oplus F^{i+1}\lrt F^{i+1}\lrt 0$ is an $S$-strongly flat precover of $\eta$. Moreover, the kernel of this precover is an $S$-weakly cotorsion complex. By pasting together all the diagrams, for every integer $i$, we get an $S$-strongly flat precover of $C$ whose kernel is $S$-weakly cotorsion.
\end{proof}

\begin{lemma}\label{lem. 4.4}
 Assume the projective dimension of $R_S$ as an $R$-module doesn't exceed $1$. Let $0\lrt C\lrt G\lrt H\lrt 0$ be a short exact sequence in $\C(\Mod R)$ such that $H$ is an $S$-strongly flat and $G$ is an $S$-weakly cotorsion complex. Then $C$ has an $S$-strongly flat precover whose kernel is $S$-weakly cotorsion.
\end{lemma}

\begin{proof}
By Proposition \ref{Prop. 4.3}, for $S$-weakly cotorsion complex $G$, there exist an $S$-strongly flat precover $F\lrt G$ with kernel $S$-weakly cotorsion. Consider the pullback diagram
\[\begin{tikzcd}
 & & 0 \dar& 0\dar\\
 & & K\rar[equals]\dar& K\dar\\
 &	0 \rar  &X\rar\dar&F\rar\dar&H\dar[equals]\rar& 0\\
 &	0\rar&C\rar\dar &G\rar\dar& H\rar & 0.\\
 & & 0 & 0
	\end{tikzcd}\]
Since $F$ and $H$ are $S$-strongly flat complexes and $\pd_R R_S\leq 1$, $X$ is $S$-strongly flat. Now it is easy to see that $X\rt C$ is an $S$-strongly flat precover of $C$ where its kernel is $S$-weakly cotorsion.
\end{proof}

\begin{lemma}\label{lem. 4.5}
 Assume the projective dimension of $R_S$ as an $R$-module doesn't exceed $1$. Let  $0\lrt B\lrt C\lrt D\lrt 0$ be a short exact sequence in $\C(\Mod R)$ such that $C$ is $S$-weakly cotorsion and $D$ has an $S$-strongly flat precover with kernel $S$-weakly cotorsion. Then $B$ has an $S$-strongly flat precover with kernel $S$-weakly cotorsion.
\end{lemma}

\begin{proof}
Let $F\lrt D$ be an $S$-strongly flat precover with kernel $S$-weakly cotorsion. Consider the pullback diagram
\[\begin{tikzcd}
 & & & 0 \dar& 0\dar\\
 & & & K\rar[equals]\dar& K\dar\\
 &	0 \rar  &B\rar\dar[equals]&G\rar\dar&F\dar\rar& 0\\
 &	0\rar & B\rar &C\rar\dar& D\rar\dar & 0.\\
 & &  & 0 & 0
	\end{tikzcd}\]
The short exact sequence $0\rt K \rt  G \rt C \rt 0$ implies that $G$ is $S$-weakly cotorsion, because $K$ and $C$ are $S$-weakly cotorsion.
By applying Lemma \ref{lem. 4.4} to the exact sequence $0\rt B\rt G\rt F\rt 0$, we obtain an $S$-strongly flat precover for $B$ whose kernel is $S$-weakly cotorsion.
\end{proof}

\begin{theorem}\label{Theorem 4.6}
 Assume the projective dimension of $R_S$ as an $R$-module doesn't exceed $1$. Then any complex of $R$-modules admits an $S$-strongly flat precover.
\end{theorem}

\begin{proof}
By Lemma \ref{lem. 4.2} it is enough to prove that any acyclic complex of $R$-modules has an $S$-strongly flat precover. Let $X$ be an acyclic complex. Since $\C(\Mod R)$ has enough injectives, there is a short exact sequence $0\rt X\rt E\rt C\rt 0$ of complexes, where $E$ is an injective complex and therefore is $S$-weakly cotorsion. By Lemma \ref{ProjDim}, the projective dimension of every $S$-strongly flat complex $F$ is at most $1$, so a similar argument as in the proof of
\cite[Theorem 4.6]{EG}, implies that $C$ is an $S$-weakly cotorsion complex. Now, given Lemma \ref{lem. 4.5}, we can construct an $S$-strongly flat precover for $X$.
\end{proof}

\begin{remark}
The authors in \cite{YL} proved that if $(\SX, \SY)$ is a complete and hereditary cotorsion pair in $\Mod R$, then the induced cotorsion pairs $(\widetilde{\SX}, dg \widetilde{\SY})$ and $(dg \widetilde{\SX}, \widetilde{\SY})$  are complete in $\C(\Mod R)$, where $\widetilde{\SX}$ is the class of all acyclic complexes of $R$-modules with all kernels in $\SX$ and $dg \widetilde{\SX}$ is the class of all complexes $X$ of $R$-modules such that all terms are in $\SX$ and $\homf(X, Y)$ is acyclic whenever $Y\in \widetilde{\SY}$. If we assume further that $\pd_R R_S\leq 1$, then complete cotorsion pair $(\SSF, \SWC)$ is hereditary. So by the completeness of the induced cotorsion pair $(\widetilde{\SSF}, dg\widetilde{\SWC})$ we get $\widetilde{\SSF}$ is precovering. But  $\widetilde{\SSF}$ is precisely the $S$-strongly flat complexes. This provides another proof for Theorem \ref{Theorem 4.6}.
\end{remark}

\begin{theorem}\label{Covering}
Let $R$ be a commutative ring and $S \subset R$ be a multiplicative subset.  Then the following are equivalent.
\begin{itemize}
\item[$(i)$] $R$ is $S$-almost perfect.
\item[$(ii)$] The class of $S$-strongly flat complexes is covering.
\end{itemize}
\end{theorem}

\begin{proof}
$(i)\Rightarrow (ii)$. Since $R$ is $S$-almost perfect, $\SSF=\Flat R$. Hence the class of $S$-strongly flat complexes and the class of flat complexes coincide. Now the result follows by \cite[Theorem 3.3]{AEGO}.

$(ii)\Rightarrow(i)$.  In order to show that $R$ is $S$-almost perfect, we need to show that $R_S$ and $R/sR$, for every $s\in S$, are perfect. To show this, by using \cite[Lemma 3.6, 3.7]{BP}, it suffices to show that every $R_S$-module and every $R/sR$-module, for every $s\in S$, admit $S$-strongly flat cover.
Let $M$ be an $R_S$-module. Let $f: F\lrt M$ be an $S$-strongly flat precover of $M$, where we consider $M$ as an $R$-module.   Consider the following commutative diagram
 \[\begin{tikzcd}
 &\overline{F}\dar &	\cdots\rar &0 \rar\dar  &F\rar{1}\dar{f}&F\rar\dar{f}&0\dar\rar&\cdots\\
&\overline{M} &	\cdots\rar& 0\rar&M\rar{1} &M\rar& 0\rar & \cdots.
	\end{tikzcd}\]
It is easy to see that $\overline{F}$ is an $S$-strongly flat precover of $\overline{M}$. By assumption, $\overline{M}$ has an $\SSF$-cover. Hence $\overline{F}$ is a summand of its cover. Therefore, every $\SSF$-cover of $\overline{M}$ is of the form $\overline{F}$.  This implies that $M$, as an $R$-module, admits an $\SSF$-cover. Hence \cite[Lemma 3.6]{BP} implies that $R_S$ is a perfect ring. By a similar argument, one can prove that, for every $s\in S$, every $R/sR$-module has an $S$-strongly flat cover. Hence \cite[Lemma 3.7]{BP} implies that $R/sR$ is a perfect ring. So $R$ is $S$-almost perfect.
\end{proof}

\section{Optimistically strongly flat complexes}\label{Sec: Homotopy Category of Optimistically Strongly Flat modules}
In this section we introduce the class of optimistically $S$-strongly flat modules, denoted by $\SOS$. We examine its homotopy category, $\K(\SOS)$, and explore the properties of the right adjoint of the inclusion $e: \K(\Prj R) \lrt \K(\SOS)$. Our study is motivated by Neeman's results on $\K(\Flat R)$ in \cite{N1} and \cite{N}.

\begin{definition}
 A flat $R$-module $F$ is called optimistically $S$-strongly flat if $F_S$ is a projective $R_S$-module and, for every $s\in S$, $F/sF$ is a projective $R/sR$-module.
\end{definition}

\begin{remark}
It is clear that $\Prj R\subset \SSF\subset\SOS\subset \Flat R$.  If $(R, S)$ is such that $(\OC)$ holds,  then $\SOS=\SSF$. Moreover, it follows from the definition that the class of optimistically $S$-strongly flat $R$-modules is closed under coproducts, extensions and kernels of epimorphisms.
\end{remark}

\begin{remark}
Let $R$ be a commutative ring and $r\in R$. Let $R[r^{-1}]$ be the localization of $R$ with respect to the multiplicative subset $S=\lbrace 1, r, r^2, r^3, \cdots\rbrace $ of $R$.  An $R$-module $C$ is called $r$-contraadjusted, if $\Ext^1_R(R[r^{-1}], C)=0$. An $R$-module $F$ is called $r$-very flat, if $\Ext^1_R(F, C)=0,$ for every $r$-contraadjusted module $C$, \cite[\S 1.6]{PS2}. In this case, the class of optimistically $S$-strongly flat modules coincides with the class of $r$-very flat modules, for instance, see \cite[Theorem 1.8]{PS2}.
\end{remark}

\begin{proposition}\label{Right Adjoint}
The natural inclusion  $e: \K(\Prj R)\lrt \K(\SOS)$  has a right adjoint $e^*: \K(\SOS)\lrt \K(\Prj R)$.
\end{proposition}

\begin{proof}
By \cite[Fact 2.8 ($i$)]{N1}, the category $\K(\Prj R)$ is well generated, and hence satisfies Brown representability. Moreover, the inclusion of $e$ respects coproducts. So by \cite[Theorem 8.4.4]{N2} it has a right adjoint.
\end{proof}
Consider the adjoint pair
\[ \xymatrix{\K(\Prj R)\ar@<0.5ex>[rr]^{e}& &\K(\SOS).\ar@<0.5ex>[ll]^{e^*} } \]

Note that since the inclusion $e$ is fully faithful and has a right adjoint $e^*$, the right adjoint $e^*$ is a Verdier quotient.
Consider the kernel of the functor $ e^*: \K(\SOS)\lrt \K(\Prj R)$; that is, the full subcategory of $\K(\SOS)$  that is annihilated by the functor $e^*$. Since $(e, e^*)$ is an adjoint pair,
\[\Ker e^* = \lbrace Y\in \K(\SOS) ~ ~\vert ~ \Hom(X, Y)=0~~~~ \forall X\in\K(\Prj R)\rbrace.\]

Following Neeman, we denote this subcategory by ${\K(\Prj R)}^{\perp}$, where orthogonal is taken in $\K(\SOS)$.  Let $i_*: \K(\Prj R)^{\perp} \lrt \K(\SOS)$ be the inclusion. By \cite[Remark 2.12]{N1}, the composition $\pi e: \K(\Prj R)\lrt  \frac{\K(\SOS)}{{\K(\Prj R)}^{\perp}}$ in the diagram
\[\xymatrix{
\K(\Prj R) \ar@<0.5ex>[dr]^e \ar@/_2.5pc/[ddr]^{\pi e}& & {\K(\Prj R)}^{{\perp}} \ar@<0.5ex>[dl]^{i_*}\\
& \K(\SOS) \ar@<0.5ex>[ul]^{e^*} \ar@<0.5ex>[d]^\pi\\
& \frac{\K(\SOS)}{\K(\Prj R)^\perp}}\]
is an equivalence of categories.

\begin{remark}
Let $\SX$ be an additive subcategory of $R$-modules such that it contains projective $R$-modules and it is closed under coproducts. Then by a similar argument as above, we can obtain a new description of $\K(\Prj R)$ as a quotient of $\K(\SX)$. In particular, if we set $\SX = \SSF$, then $\K(\Prj R)$ is equivalent to a quotient of $\K(\SSF)$.
\end{remark}

In the following we study the subcategory $\K(\Prj R)^{\perp}$ of $\K(\SOS)$.

\begin{theorem}\label{Theorem 8.6}
Let $Z$ be an object of $\K(\SOS)$. The following are equivalent.
\begin{itemize}
\item[$(i)$] $Z$ lies in the subcategory $\K(\Prj R)^{\perp}\subset \K(\SOS)$.
\item[$(ii)$] $Z$ is a filtered direct limit of contractible complexes of projective $R$-modules.
\item[$(iii)$] $Z$ is an acyclic complex of optimistically $S$-strongly flat $R$-modules
 \[\begin{tikzcd}
 &\cdots\rar& Z^{i-1}\rar{\delta^{i-1}}&Z^{i}\rar{\delta^i}& Z^{i+1} \rar& \cdots
	\end{tikzcd}\]
such that the kernels $K^i$ of the maps $\delta^i$ are all optimistically $S$-strongly flat $R$-modules.
\end{itemize}
\end{theorem}

\begin{proof}
$(i) \Longrightarrow (ii)$ Let $Z\in\K(\Prj R)^{\perp}$. Since $\K(\SOS)\subset \K(\Flat R)$, $Z\in \K(\Flat R)$ and every map from an object $X\in\K(\Prj R)$ to $Z$ is null homotopic. Now \cite[Lemma 8.5]{N1} implies that the complex $Z$ can be written as a filtered colimit of contractible complexes of finitely generated projective $R$-modules.

$(ii)\Longrightarrow (iii)$ Let $Z$ be a filtered direct limit of contractible complexes of projective $R$-modules. Then the complex $Z$ is acyclic. Moreover, since for every contractible complex $(C, \delta)$ of projective $R$-modules, the sequence $0\lrt \Ker \delta^i\lrt C^i\lrt \Ker\delta^{i+1}\lrt 0$  is split exact for each $i$, $\Ker \delta^i$, for every integer $i$, is a projective $R$-module. Therefore, the kernel of $\delta^i: Z^i\lrt Z^{i+1}$, for every $i$, is a filtered direct limit of projective modules, and hence is flat. In order to show that all kernels are optimistically $S$-strongly flat $R$-modules, it remains to show that, for each $i$, $\Ker\delta^i\otimes_R R_S$ and $\Ker\delta^i\otimes_R R/sR$, for every $s\in S$, are projective as $R_S$-module and $R/sR$-module, respectively. To see this, we note that since the module $Z^i$ for each $i$, is optimistically $S$-strongly flat $R$-module, so $Z^i\otimes_R R_S$ and $Z^i\otimes_R R/sR$, for every $s\in S$, are projective. Hence we obtain an acyclic complex
\[\xymatrix{\cdots \ar[r] & Z^{i-1}\otimes_R R_S \ar[rr]^{\delta^{i-1}\otimes_R R_S} && Z^i\otimes_R R_S \ar[rr]^{\delta^i\otimes_R R_S} && Z^{i+1}\otimes_R R_S \ar[r] & \cdots} \]
of projective $R_S$-modules such that all kernels are flat $R_S$-modules and, for every $s$ in $S$, an acyclic complex
\[\xymatrix{\cdots \ar[r] & Z^{i-1}\otimes_R R/sR \ar[rr]^{\delta^{i-1}\otimes_R R/sR} && Z^i\otimes_R R/sR \ar[rr]^{\delta^i\otimes_R R/sR} && Z^{i+1}\otimes_R R/sR  \ar[r] & \cdots}\]
of projective $R/sR$-modules with all kernels flat $R/sR$-modules. Now \cite[Remark 2.15]{N1} implies that these complexes are contractible and hence, for each $i$, $\Ker\delta^i\otimes_R R_S$ and $\Ker\delta^i\otimes_R R/sR$, for every $s\in S$, are projective.

$(iii)\Longrightarrow(i)$ Since $\SOS\subset \Flat R$, so by the same argument as in the proof of $(iii) \Rightarrow (i)$ of \cite[Theorem 8.6]{N1} we get the result.
\end{proof}

\begin{remark}
The complex $Z$ in the above theorem is called an optimistically $S$-strongly flat complex, compare Definition \ref{DefinitionComplexes}.
\end{remark}

By a similar argument as in the proof of Theorem \ref{Theorem 8.6}, we get the following. So we omit the proof.

\begin{theorem}\label{Describtion as subcat of K(SSF)}
 Let $Z$ be an object of $\K(\SSF)$. The following are equivalent.

\begin{itemize}
\item[$(i)$] $Z$ lies in the subcategory $\K(\Prj R)^{\perp}\subset \K(\SSF)$.

\item[$(iii)$] $Z$ is an acyclic complex of  $S$-strongly flat $R$-modules
 \[\begin{tikzcd}
 &\cdots\rar& Z^{i-1}\rar{\delta^{i-1}}&Z^{i}\rar{\delta^i}& Z^{i+1} \rar& \cdots
	\end{tikzcd}\]
such that the kernels $K^i$ of the maps $\delta^i$ are optimistically $S$-strongly flat $R$-modules.
\end{itemize}
\end{theorem}

\begin{corollary}\label{Characterization}
Let $R$ be a commutative ring and $S\subset R$ be a multiplicative subset such that $(\OC)$ holds. Let $Z$ be an object of $\K(\SSF)$. The following are equivalent.
\begin{itemize}
\item[$(i)$] $Z$ lies in the subcategory $\K(\Prj R)^{\perp}\subset \K(\SSF)$.

\item[$(ii)$] $Z$ is an $S$-strongly flat complex, that is, $Z$ is an acyclic complex of  $S$-strongly flat $R$-modules
 \[\begin{tikzcd}
 &\cdots\rar& Z^{i-1}\rar{\delta^{i-1}}&Z^{i}\rar{\delta^i}& Z^{i+1} \rar& \cdots
	\end{tikzcd}\]
such that the kernels $K^i$ of the maps $\delta^i$ are all $S$-strongly flat $R$-modules.
\end{itemize}
\end{corollary}

\begin{proof}
Note that since $(\OC)$ holds, then $\SSF=\SOS$. Now the result follows by Theorem \ref{Theorem 8.6}.
\end{proof}

\begin{corollary}\label{PrecoverOverMatlisDomain}
Let $R$ be a commutative ring and $S\subset R$ be a multiplicative subset such that $(\OC)$ holds. Then every complex of $R$-modules has a $\K(\Prj R)^{\perp}$-precover.
\end{corollary}

\begin{proof}
By the previous corollary, objects in $\K(\Prj R)^{\perp}$ are precisely $S$-strongly flat complexes. So the result follows by Theorem \ref{Theorem 4.6}.
\end{proof}

In the rest of this section, we study the behavior of $\K(\Prj R)^{\perp}\subset \K(\SOS)$ under tensor product.

\begin{lemma}(\cite[Proposition 9.2]{N1})\label{Prop. 9.2}
Let $Z$ be an object in $\K(\Prj R)^{\perp}\subset \K(\SOS)$ and $X$ be an arbitrary complex of $R$-modules. Then, for every finitely presented module $M$, the complex $\homf(M, X\otimes Z)$ is acyclic. In fact, the natural map of complexes \[\homf(M, X)\otimes Z\lrt \homf(M, X\otimes Z)\] is an isomorphism.
\end{lemma}

\begin{proof}
Since for every $i$, the module $Z^i$ is a flat $R$-module, by the same argument as in the proof of \cite[Proposition 9.2]{N1} we get the result.
\end{proof}

\begin{proposition}\label{Corollary 9.4}
Let $Z$ be an object of $\K(\SOS)$. The following are equivalent.
\begin{itemize}
\item[$(i)$] $Z$ lies in the subcategory $\K(\Prj R)^{\perp}\subset \K(\SOS)$.
\item[$(ii)$] The complex $X\otimes Z$ is acyclic, for all arbitrary complex $X$ of $R$-modules.
\item[$(iii)$] The complex $M\otimes Z$ is acyclic, for all $R$-module $M$.
\item[$(iv)$]  The complex $\homf(N, Z)$ is acyclic, for all finitely presented $R$-module $N$.
\end{itemize}
\end{proposition}

\begin{proof}
$(i)\Longrightarrow (ii)$ Since by Theorem \ref{Theorem 8.6}, $Z$ is filtered colimit of contractible complexes $Z_\lambda$. So $X\otimes Z$ is filtered colimit of the contractible complexes $X\otimes Z_\lambda$. Hence is acyclic.

$(ii)\Longrightarrow (iii)$  is trivial.

$(iii)\Longrightarrow (iv)$ Let $N$ be a finitely presented $R$-module. By Lemma \ref{Prop. 9.2}, we have an isomorphism
\[\Hom_R(N, R)\otimes Z\st{\simeq}\lrt \homf(N,  Z).\]
By the assumption $(iii)$, $\Hom_R(N, R)\otimes Z$ is acyclic, so is $\homf(N,  Z)$.

$(iv)\Longrightarrow (i)$ Since $\SOS\subset \Flat R$ and therefore $Z$ is a complex of flat $R$-module such that for every finitely presented $R$-module $N$, $\homf(N, Z)$ is acyclic,  so \cite[Lemma 9.3]{N1} implies that $Z$ is an acyclic complex of flat $R$-modules with all kernels of maps $\delta^i: Z^i\rt Z^{i+1}$ flat. In order to complete the proof, by Theorem \ref{Theorem 8.6}, we need to show that all kernels are optimistically $S$-strongly flat $R$-modules. To show this, we just do as we did at the end of the proof $(ii)\Rightarrow (iii)$ of Theorem \ref{Theorem 8.6}.
\end{proof}

\begin{corollary}\label{Closed Under Tensor}
Let $X$ be a complex in $\K(\SOS)$.
\begin{itemize}
\item[$(i)$]
Let $Z$ be an object in $\K(\SOS)$, then  $X\otimes Z$ is in $\K(\SOS)$.
\item[$(ii)$]
Let $Z$ be in $\K(\Prj R)^{\perp}\subset \K(\SOS)$, then  $X\otimes Z$ is in $\K(\Prj R)^{\perp}$.
\end{itemize}
\end{corollary}

\begin{proof}
$(i)$ Since the class of optimistically $S$-strongly flat $R$-modules are closed under coproducts, then by the definition of the tensor product of complexes we get the result.

$(ii)$ First we note that by $(i)$, $X\otimes Z$ is a complex of optimistically $S$-strongly flat $R$-modules. Lemma \ref{Prop. 9.2} implies that $\homf(M, X\otimes Z)$ is acyclic. Now $(iv)\Rightarrow( i)$ of Proposition \ref{Corollary 9.4} implies that $X\otimes Z$ lies in $\K(\Prj R)^{\perp}$.
\end{proof}

\begin{remark}\label{Tensor ideal}
The statement $(i)$ of the above corollary, implies that the category $\K(\SOS)$ is a tensor triangulated category and the statement $(ii)$, implies that the subcategory $\K(\Prj R)^{\perp}\subset \K(\SOS)$ is a tensor ideal, for definitions see \cite{B1}, \cite{B2}. Moreover, if $(R, S)$ is a ring such that $(\OC)$ holds, then $\K(\SSF)$ is a tensor ideal.
\end{remark}

\section{Existence of the right adjoint}\label{Sec: Some Adjoints in Homotopy Category}
In Section \ref{Sec: Homotopy Category of Optimistically Strongly Flat modules}, $\K(\Prj R)$ is described as a Verdier quotient $\K(\SOS)/{\K(\Prj R)}^{\perp}$. If $(\OC)$ holds, $\SOS =\SSF$. In this section, we show that the quotient map $e^*:\K(\SSF)\lrt \K(\SSF)/{\K(\Prj R)}^{\perp}$ has a right adjoint $e_*:\K(\SSF)/{\K(\Prj R)}^{\perp} \lrt \K(\SSF)$  provided $(\OC)$ holds. By Corollary \ref{Characterization}, in this case, ${\K(\Prj R)}^{\perp}$ is the class of all $S$-strongly flat complexes; that is, acyclic complexes whose all kernels are $S$-strongly flat $R$-modules. This class will be denoted by $\sS$.

To this end, we need to recall some known facts on the existence of adjoints that we need; for details and proofs see \cite[Theorems 9.1.13, 9.1.16 and Proposition 9.1.18]{N2}.
Recall that a full subcategory $\SU$  of a triangulated category $\ST$ is called thick if it is triangulated and is closed under direct summands.

 \begin{facts}\label{Facts}
 Let $\ST$ be a triangulated category and $\SU$ be a thick subcategory of $\ST$. Then the following are equivalent.
\begin{itemize}
\item[$(i)$] The inclusion functor $\SU\lrt \ST$ has a right adjoint.
\item[$(ii)$] For each $X\in \ST$, there exists a triangle \[X'\lrt X\lrt X''\lrt \Sigma X'\] such that $X'$ in $\SU$ and $X''$ belongs to the right orthogonal of $\SU$ in $\ST$; that is  $X''$ is an object in $\ST$ and $\Hom(Y, X'')=0$ for all $Y\in\SU$.
\item[$(iii)$] The quotient functor $\ST\lrt \ST/\SU$ has a right adjoint.
\item[$(iv)$] The composition $\SU^\perp\lrt \ST\lrt \ST/\SU$ is an equivalence.
\end{itemize}
 \end{facts}

We also need the following fact, see \cite[Proposition 1.4]{N}.

\begin{proposition}\label{Proposition 1.4}
Let $\ST$ be a triangulated category and $\SU$ be a thick subcategory of $\ST$. Further, assume that
\begin{itemize}
\item[$(i)$] Every object in $\ST$ has an $\SU$-precover.
\item[$(ii)$] Every idempotent in $\ST$ splits.
\end{itemize}
Then the inclusion $\SU\lrt \ST$ has a right adjoint.
\end{proposition}

Back to our case, let $e^*:\K(\SSF)\lrt \K(\SSF)/{\K(\Prj R)}^{\perp}$ be the Verdier quotient map. The facts above imply that in order to show the map $e^*$ has a right adjoint it is enough to show that the inclusion map $i_*: \K(\Prj R)^{\perp}\lrt \K(\SSF)$ has a right adjoint.

\begin{proposition}
Let $R$ be a commutative ring and $S\subset R$ be a multiplicative subset such that $(\OC)$ holds. Then the inclusions of $ \K(\Prj R)^{\perp}$ into $\K(\SSF)$ and $\K(\Mod R)$,  have  right adjoints.
\end{proposition}

\begin{proof}
We use Proposition \ref{Proposition 1.4} to show the result. We need to show that the conditions of Proposition \ref{Proposition 1.4} are satisfied. Since $(\OC)$ holds, Corollary \ref{PrecoverOverMatlisDomain} implies that every object in $\K(\Mod R)$ has a $\K(\Prj R)^{\perp}$-precover and therefore every object in $\K(\SSF)\subset\K(\Mod R)$ has a  $\K(\Prj R)^{\perp}$-precover. Moreover, since the categories $\K(\Prj R)^{\perp}, \K(\SSF)$ and $\K(\Mod R)$ all have coproducts, by \cite[Proposition 1.6.8]{N2} we conclude that all idempotents split in theses categories. Hence we have the conditions of Proposition \ref{Proposition 1.4}.
So the inclusions of $ \K(\Prj R)^{\perp}$ into $\K(\SSF)$ and $\K(\Mod R)$,  have  right adjoints.
\end{proof}

The above Proposition and Facts \ref{Facts} imply the following.

\begin{theorem}\label{Existence of adjoint of quotient}
Let $R$ be a commutative ring and $S\subset R$ be a multiplicative subset such that $(\OC)$ holds. Then the functor $e^*:\K(\SSF)\lrt \K(\Prj R)$ has a right adjoint $e_*:\K(\Prj R)\lrt \K(\SSF)$.
\end{theorem}

\begin{remark}
Since $e^*$ is a Verdier quotient map, its right adjoint $e_*$ is fully faithful. Hence $e_*$ gives a non-obvious embedding of $\K(\Prj R)$ into $\K(\SSF)$.
\end{remark}

Using a similar argument as in the proof of \cite[Theorem 3.2]{N1}, we obtain the following theorem. Therefore, we will omit the proof.

\begin{theorem}\label{Existence of adjoint of inclusion}
 Let $R$ be a commutative ring and $S\subset R$ be a multiplicative subset such that $(\OC)$ holds. Then the inclusion functor $\K(\SSF)\lrt \K(\Mod R)$ has a right adjoint.
\end{theorem}

Let $R$ be a commutative ring. Neeman in \cite{N1}, in order to prove that an object in $\K(\Flat R)$ has a flat precover, considered an auxiliary ring
$T=T(R)$ and two functors $\inc:\C(\Mod R)\lrt \Mod T$  and $C: \Mod T\lrt \C(\Mod R)$. Although we do not use Neeman's method in our setting, studying the behavior of strongly flat complexes and optimistically strongly flat complexes under these two functors are interesting. In the following we do this. First, let us recall the ring $T$ and two functors $\inc$ and $C$, briefly. Let
\[\begin{tikzcd}
 & \cdots \ar{r}{\delta^{i-2}} &\cdot^{i-1}
\ar{r}{\delta^{i-1}}& \cdot^i\ar{r}{\delta^{i}}& \cdot^{i+1}\ar{r}{\delta^{i+1}}& \cdot^{i+2}\ar{r}{\delta^{i+2}}& \cdots
	\end{tikzcd}\]
be the quiver with relations $\delta^{i+1}\delta^i=0$, for all $i \in \Z$. Let $T$ be the free $R$-module with basis $\lbrace 1, \delta^i, e^j\rbrace$, with $i, j\in\mathbb{Z}$. Since $R$ commutes with all the basis elements, it is an algebra.
 The relations between the basis elements assert that $1$ is the identity element and the set $\{e^j\}_{j \in \Z}$ are orthogonal idempotents.
Then he defined the inclusion functor
\[\inc: \C(\Mod R)\lrt \Mod T\]
such that takes a complex $Z$ to $\bigoplus_{i=-\infty}^{\infty} Z^i$ with a natural $T$-module structure. This functor is fully faithful and has a right adjoint, which is denoted  by
\[C: \Mod T\lrt \C(\Mod R).\]
The functor $C$ takes the $T$-module $M$ to the complex
\[\begin{tikzcd}
 &\cdots\rar{\delta^{i-2}}& e^{i-1}M\rar{\delta^{i-1}}&e^{i}M\rar{\delta^i}& e^{i+1}M \rar{\delta^{i+1}}& \cdots.
	\end{tikzcd}\]

In the following we collect some facts about the functors $\inc$ and $C$, all from \cite[\S 2]{N1}.
\begin{facts}
The following hold true.
\begin{itemize}
\item[$(i)$] The functor $C$ is exact and preserves colimits.
\item[$(ii)$] If $P$ is a projective $T$-module, then the complex $C(P)$ is a contractible complex of projective $R$-modules.
\item[$(iii)$] If $F$ is a flat $T$-module, then  the complex $C(F)$ is a flat complex of $R$-modules.
\item[$(iv)$] If $Z$ is a contractible complex of projective $R$-modules, then $\inc(Z)$ is a projective $T$-module.
\item[$(v)$] If $Z$ is a flat complex, then $\inc(Z)$ is a flat $T$-module.
\end{itemize}
\end{facts}

\begin{remark}
Let $\sS$ be the class of $S$-strongly flat complexes, which are defined as acyclic complexes where all kernels are $S$-strongly flat $R$-modules. We set $\inc(\sS) = \SF$. Since every $S$-strongly flat complex is also a flat complex, it follows that $\SF \subset \Flat T$. Furthermore, because every contractible complex of projectives is an $S$-strongly flat complex, we can conclude that $\Prj T \subset \SF$.
\end{remark}

We show that the class $\SF\subset \Mod T$ is a precovering class provided $\pd_R R_S\leq 1$.

\begin{proposition}
Let $R$ be a commutative ring and $S\subset R$ be a multiplicative subset such that projective dimension of $R_S$ as an $R$-module doesn't exceed $1$. Then every object in the category of $T$-modules has an $\SF$-precover.
\end{proposition}

\begin{proof}
Let $M$ be a $T$-module. Then $C(M)$ is in $C(\Mod R)$. By Theorem \ref{Theorem 4.6} we know that $C(M)$ has an $S$-strongly flat precover $F$. Let $F\lrt C(M)$ be an $S$-strongly flat precover. By applying the functor $\inc$ we get $\inc(F)\lrt \inc(C(M))$. But we know that there is a map from $\inc(C(M))$ to $M$, so we get a map from  $\inc(F)$ to $M$. We claim that the map $\inc(F)\lrt M$ is an $\SF$-precover. First we note that $\inc(F)$ is in the class $\SF$.  We need to prove it a precover.

Assume that we are given a map $\inc(F') \lrt M$, where $F'$ is in $\sS$. So we have a map $F'\cong C (\inc(F'))\lrt C(M)$. It must factor through the $S$-strongly flat precover $F\lrt C(M)$. Therefore, we have a factorization
\[ \begin{tikzcd}
 & F'\rar & F\rar& C(M).
	\end{tikzcd}\]
Hence we have a factorization
\[ \begin{tikzcd}
 & \inc(F')\rar & \inc(F)\rar& \inc(C(M))\rar& M.
	\end{tikzcd}\]
\end{proof}

In the following we investigate the behaviour of the class of optimistically $S$-strongly flat complexes, denoted by $\SO$, and $\CF:= \inc(\SO)$, under the functors $\inc$  and $C$, respectively.

\begin{proposition}
Let $\SO$ denote the class of optimistically $S$-strongly flat complexes. Set $\CF=\inc(\SO)$.
\begin{itemize}
  \item [$(1)$] Let $Z \in \SO$. Then $\inc(Z) \in \Flat T \cap (\SOS).$
  \item [$(2)$] Let $F\in\CF$. Then $C(F)$ is an acyclic complex with all kernels optimistically $S$-strongly flat $R$-modules.
\end{itemize}
\end{proposition}

\begin{proof}
$(1).$ Let $Z$ be an optimistically $S$-strongly flat complex, that is, $Z$ is an acyclic complex such that all kernels are optimistically $S$-strongly flat.
Since $Z$ is a flat complex, that is an acyclic complex with kernels flat, then Fact \ref{Facts}$(v)$ implies that $\inc(Z)$ is a flat $T$-module. Moreover, since the class of optimistically $S$-strongly flat $R$-modules is closed under coproducts, by the definition of the functor $\inc$ we get $\inc(Z)$ is an optimistically $S$-strongly flat $R$-module. Hence the result follows.

$(2).$ Let $F\in\CF$. Then $C(F)$ is the complex
\[\begin{tikzcd}
 &\cdots\rar& e^{i-1}F\rar{\delta^{i-1}}&e^{i}F\rar{\delta^i}& e^{i+1}F \rar& \cdots.
	\end{tikzcd}\]
Since $F$ is a flat $T$-module, Fact \ref{Facts}$(iii)$ implies that $C(F)$ is a flat complex of $R$-modules, that is an acyclic complex with all kernels flat $R$-modules. Further, since $F$ is an optimistically $S$-strongly flat $R$-module, decomposition $F=e^i F\oplus (1-e^i)F$ implies that $e^i F\otimes_R R_S$ and $e^i F\otimes_R R/sR$, for every $s\in S$, are projective $R$-modules. Therefore, all terms in the acyclic complex $C(F)$ are optimistically $S$-strongly flat $R$-modules and all kernels are flat $R$-modules. Now \cite[Remark 2.15]{N1} tells us that all kernels are optimistically $S$-strongly flat $R$-modules. Hence $C(F)$ is an acyclic complex with all kernels optimistically $S$-strongly flat $R$-modules, that is, $C(F)\in\SO$.
\end{proof}

\section{Almost well generated triangulated category}\label{Sec: Almost Well Generated}
In this section, we define and study the notion of $S$-almost well generated triangulated categories. We show that if $R$ is an $S$-almost perfect ring, then $\K(\Flat R)$ is $S$-almost well generated. We prove the converse in special cases. We do not know if it holds in general.

Let $\ST$ be a triangulated category with coproducts. Recall that $\ST$ is well generated if it has an $\alpha$-perfect set of $\alpha$-small generators, for
some regular cardinal $\alpha$. Such a set is called an $\alpha$-compact generating set of $\ST$ \cite{N1}. If $\alpha = \aleph_0$, then $\ST$ is called compactly generated. Neeman in \cite{N1} proved that $\K(\Prj R)$ is always well generated.

For a subcategory $\SX$ of $R$-modules, we let $\Add(\SX)$ be the class of all direct summands of arbitrary direct sums of objects from $\SX$.

\begin{definition}
Let $\SX$ be an additive full subcategory of $R$-modules such that it is closed under coproducts. We say that $\K(\SX)$, the homotopy category of $\SX$,  is  $S$-almost well generated if $\K(\Add(\SX\otimes_R R_S))$ and $\K(\Add(\SX\otimes_R R/sR))$, for every $s\in S$, are well generated.
\end{definition}

\begin{example}\label{Examples}
Let $R$ be a commutative ring and $S \subset R$ be a multiplicative subset.
\begin{itemize}
\item[$(i)$] Assume  $\SX=\Prj R$, then $\K(\Prj R)$ is $S$-almost well generated. We have
\[\Add(\Prj R\otimes_R R_S)=\Prj R_S,\]
and for every $s\in S$,
\[\Add(\Prj R\otimes_R R/sR)=\Prj R/sR.\]

Now  since $\K(\Prj R_S)$ and $\K(\Prj R/sR)$ are well generated, so $\K(\Prj R)$ is $S$-almost well generated.

\item[$(ii)$] Assume $\SX=\SSF$,  then $\K(\SSF)$ is $S$-almost well generated. To see this, by \cite[Lemma 3.1]{BP}  and the fact that $\Prj R\subset \SSF$ we have inclusions  \[\Prj R\otimes_R R_S\subset\SSF\otimes_R R_S\subset \Prj R_S,\] and for every $s\in S$, \[\Prj R\otimes_R R/sR\subset \SSF\otimes_R  R/sR\subset\Prj R/sR.\]
Now $(i)$ tells us that $\Add(\SSF\otimes_R R_S)=\Prj R_S$, and for every $s\in S$, $\Add(\SSF\otimes_R R/sR)=\Prj R/sR$ and hence the homotopy categories of these subcategories are well generated.

\item[$(iii)$] Assume $\SX=\SOS$, then $\K(\SOS)$ is $S$-almost well generated. We note that by the definition of optimistically $S$-strongly flat  modules we have $\SOS\otimes_R R_S\subset \Prj R_S$ and for every $s\in S$,  $\SOS\otimes_R R/sR\subset \Prj R/sR$. Moreover, since  $\SSF\subset \SOS$ there are inclusions
\[ \SSF\otimes_R R_S\subset \SOS\otimes_R R_S\subset \Prj R_S,\]
and for every $s\in S$,
\[\SSF\otimes_R R/sR \subset \SOS\otimes_R R/sR\subset \Prj R/sR.\]
Now $(ii)$ implies that $\Add(\SOS\otimes_R R_S)=\Prj R_S$, and for every $s\in S$, \[\Add(\SOS\otimes_R R/sR)=\Prj R/sR.\] Hence the homotopy categories of these subcategories are well generated.
\end{itemize}
\end{example}

\begin{theorem}\label{S-Almost WG}
Let $R$ be a commutative ring and $S \subset R$ be a multiplicative subset. Consider the following statements.
\begin{itemize}
\item[$(i)$]  $R$ is $S$-almost perfect.
\item[$(ii)$] $\K(\Flat R)=\K(\SSF)$.
\item[$(iii)$]  $\K(\Flat R)$ is $S$-almost well generated.
\end{itemize}
Then $(i)$ is equivalent to $(ii)$ and $(ii)$ implies $(iii)$.
\end{theorem}

\begin{proof}
$(i)\Longleftrightarrow (ii)$. Ring $R$ is $S$-almost perfect if and only if $\Flat R=\SSF$ if and only if $\K(\Flat R)=\K(\SSF)$.

$(ii)\Longrightarrow (iii)$. Since by Example \ref{Examples} $(ii)$,  $\K(\SSF)$ is $S$-almost well generated, so is $\K(\Flat R)$.
\end{proof}

In the following we investigate special rings such that $\K(\Flat R)$ being $S$-almost well generated implies that $R$ is $S$-almost perfect.

\begin{lemma}\label{FlatIsoFlatR_SFlatR/sR}
Let $R$ be a commutative ring and $S \subset R$ be a multiplicative subset.
\begin{itemize}
\item[$(i)$] $\K(\Flat R\otimes_R R_S) = \K(\Flat R_S)$.
\item[$(ii)$] Assume that for every $s\in S$, the ring $R/sR$ is perfect. Then\[\K(\Add(\Flat R \otimes_R R/sR))=\K(\Flat R/sR).\]
\end{itemize}
\end{lemma}

\begin{proof}
 $(i)$ We show that $\Flat R\otimes_R R_S= \Flat R_S$, so the result follows.
Note that  $\Flat R\otimes_R R_S\subset \Flat R_S$. Now let $X$ be an object in $\Flat R_S$. Since $R_S$ is a flat $R$-module, then $X$ is a flat $R$-module. Moreover, since  $X$ is $R_S$-module, we have  $X \simeq X\otimes_R R_S$. Hence $X$ is an object in $\Flat R\otimes_R R_S$ and therefore $\Flat R_S\subset \Flat R\otimes_R R_S$.

$(ii)$  Note that, for every $s\in S$, $\Flat R\otimes_R R/sR\subset \Flat R/sR$.  Now let $X\in \Flat R/sR$. Since  for every $s\in S$, $R/sR$ is perfect so we have $\Flat R/sR=\Prj R/sR$. Then Example \ref{Examples}$(i)$ implies that, for every $s\in S$,  $\Prj R/sR=\Add(\Prj R\otimes_R R/sR)$, so $X$ is in $\Add(\Prj R\otimes_R R/sR)\subset \Add(\Flat R\otimes_R R/sR)$.
\end{proof}

\begin{theorem}\label{Krull Dimension}
Let $R$ be a commutative Noetherian ring of Krull dimension $1$ and $S\subset R$ be a multiplicative subset consisting of some nonzero-divisors in $R$. Then $R$ is $S$-almost perfect if and only if $\K(\Flat R)$ is $S$-almost well generated.
\end{theorem}

\begin{proof}
By Theorem \ref{S-Almost WG}, it is enough to show that if $\K(\Flat R)$ is $S$-almost well generated, then $R$ is $S$-almost perfect. Since $\K(\Flat R)$ is $S$-almost well generated, $\K(\Flat R\otimes_R R_S)$ and $\K(\Add(\Flat R\otimes_R R/sR))$  are well generated. Lemma \ref{FlatIsoFlatR_SFlatR/sR} $(i)$ implies that $\K(\Flat R_S)$  is well generated. Hence by using \cite[Theorem 5.2]{St}, $R_S$ is perfect. To complete the proof, we need to show that for every $s\in S$, $R/sR$ is also perfect. To see this, we just note that the assumption on the ring $R$ implies that the quotient ring $R/sR$, for every $s\in S$, is Artinian and hence is perfect.
\end{proof}

\begin{remark}
Assume that $\K(\Flat R)$ is $S$-almost well generated. By the proof of Theorem \ref{Krull Dimension}, we conclude that the ring $R_S$ is always perfect. However, for arbitrary ring $R$, we do not know whether the quotient ring $R/sR$, for  $s \in S$, is perfect.
\end{remark}

\section*{Acknowledgments}
The authors thank Peter Jørgensen and Amnon Neeman for answering questions and their insightful comments. The work of the first author is based on research funded by Iran National Science Foundation (INSF) under project No. 4001480. The research of the second author was supported by a grant from IPM. This work is partly done during a visit of the authors to the Institut des Hautes \'{E}tudes Scientifiques (IHES), Paris, France. They would like to express their gratitude for the support and excellent atmosphere at IHES.

\end{document}